\documentclass[11pt,a4paper,reqno]{amsart}
\usepackage{pstricks,pst-node}
\usepackage{a4,amssymb}
\usepackage{mathrsfs}
\newpsobject{showgrid}{psgrid}{subgriddiv=1,griddots=5,gridlabels=6pt}

\newtheorem{thm}{Theorem}[section]
\newtheorem{prop}[thm]{Proposition}
\newtheorem{lem}[thm]{Lemma}
\newtheorem{cor}[thm]{Corollary}

\theoremstyle{definition}

\newcommand{\R}{\ensuremath{\mathscr{R}}}

\renewcommand{\L}{\ensuremath{\mathscr{L}}}
\newcommand{\D}{\ensuremath{\mathscr{D}}}
\newcommand{\J}{\ensuremath{\mathscr{J}}}
\renewcommand{\H}{\ensuremath{\mathscr{H}}}
\newcommand{\II}{\ensuremath{\mathscr{I}}}

\newcommand{\sub}{\subseteq}

 \newcommand{\e }{\varepsilon}
\newcommand{\s }{\sigma}
\renewcommand{\a}{\alpha}
\renewcommand{\b }{\beta}
\renewcommand{\d }{\delta}

\newcommand{\f }{\varphi}
\newcommand{\h }{\theta}

\newcommand{\G }{\Gamma}
\renewcommand{\leq}{\leqslant}
\renewcommand{\geq}{\geqslant}
\renewcommand{\t }{\tau}

\newcommand{\dom}{\operatorname{dom}}
\newcommand{\PG}{\operatorname{PG}}
\renewcommand{\r }{\rho}
\newcommand{\im}{\operatorname{im}}
\newcommand{\Com}{\operatorname{Com}}

\begin{document}
%\baselineskip{6mm}

\title{Graph products of right cancellative monoids}
\date{\today}
\maketitle

\begin{center}
  John Fountain \\
   Department of Mathematics,
  University of York,\\
  Heslington, York YO10 5DD, U.K.\\
  {\bf e-mail} : {\tt jbf1@york.ac.uk}\\
    $\phantom{e}$\\

  Mark Kambites \\
School of Mathematics, University of Manchester\\
Manchester M13 9PL, U.K.\\
{\bf e-mail} :  {\tt Mark.Kambites@manchester.ac.uk}\\
  $\phantom{e}$\\
\end{center}

\begin{abstract}Our first main result shows that a graph
product of right cancellative monoids is itself right cancellative. If
each of the component monoids satisfies the condition that the
intersection of two principal left ideals is either principal or empty,
then so does the graph product. Our second
main result gives a presentation for the inverse hull
of such a graph product.
 We then specialise to
the case of the inverse hulls of graph monoids, obtaining what we call
polygraph monoids. Among other properties, we observe that polygraph
monoids are $F^*$-inverse. This follows from a general
characterisation of those right cancellative monoids with inverse hulls
that are $F^*$-inverse. 
\end{abstract}

\section*{Introduction}
Graph products of groups were introduced by E.\ R.\ Green in her 
thesis \cite{green} and have since been studied by several authors, for
example, \cite{hermiller} and \cite{crisplaca}. In these two papers,
passing reference is made to graph products of monoids, which are
defined in the same way as graph products of groups and have been
studied specifically by, among others, Veloso da Costa, and Fohry and Kuske
\cite{costa1,costa2,fohry}.  

In this paper we are interested in graph products of right cancellative
monoids. Free products and restricted direct products are special cases of graph
products, and a free or (restricted) direct product of right cancellative
monoids is again right cancellative. In Section~\ref{graph products}, in our first main result, we
generalise these observations to obtain a corresponding result for graph
products. 

We then concentrate on right cancellative monoids in which the
intersection of two principal left ideals is either principal or empty.
Following the terminology from ring theory (see for example
\cite{beauregard}) we call these monoids \textit{left LCM monoids}.  A useful concept in the study of these
monoids is 
 the notion of the inverse hull of a right
cancellative monoid. In Section~\ref{inverse hulls}, after generalities on inverse hulls, we give several
(known)  characterisations of inverse hulls of left LCM monoids and use them to show 
 that a graph
product of left LCM monoids is itself left LCM. We then consider
presentations for inverse hulls of graph products
of left LCM monoids.
In Section~\ref{polygraph monoids} we
specialise the presentation to the case where each component monoid is
free on one generator, we obtain what we call 
polygraph monoids, generalising the polycylic monoids discussed in
\cite[Chapter~9]{lawson}. 

In the final section, we concentrate on left LCM monoids with two-sided
cancellation. Among these monoids we
characterise those with an inverse hull that is $F^*$-inverse
(see Section~\ref{F*-inverse} for the definition), and observe that, in
particular, polygraph monoids are $F^*$-inverse.

We assume that the reader is familiar with the basic ideas of semigroup
theory  (see,
for example, \cite{cp61,howie,lawson}).

\section{Graph products} \label{graph products}

 For us, a \textit{graph} $\Gamma =(V,E)$ is a set $V$ of
\textit{vertices} together with an irreflexive, symmetric relation 
$E \sub  V \times V$ whose elements are called
\textit{edges}. In particular, $\Gamma$ is loop free. 
We say that $u$ and $v$ are \textit{adjacent} in $\Gamma$
if $(u,v) \in E$. For each $v\in V$, let $M_v$ be a monoid; whenever
necessary we can, without loss of generality, assume the monoids $M_v$
are disjoint. We denote the free product of the $M_v$ by $\prod^{\star}M_v$
and write $x \centerdot y$ for the product of $x,y\in \prod^{\star}M_v$.

We define the \textit{graph product} $\Gamma_{v\in V}M_v$ of the
$M_v$ to be the quotient of  $\prod^{\star}M_v$ factored by the
congruence generated by the relation
 $$R_{\Gamma} =  \{ (m\centerdot n , n\centerdot m) : m\in M_u, n\in M_v \text{ and } u,v \text{ are
adjacent in } \Gamma \}.$$  

Alternatively, if for each $M_v$ we have a presentation $\langle
A_v \mid R_v \rangle$, 
then $\Gamma_{v\in V}M_v$  is the monoid with presentation 
$\langle A \mid R\rangle$ where 
$$A = \bigcup_{v\in V}A_v  \text{ \ and \ } 
R = \bigcup_{(u,v)\in E}\{ ab=ba: a\in A_u, b\in A_v\} \ \cup \ \bigcup_{v\in V}R_v.$$

For the rest of this section we will write $M$ for $\Gamma_{v\in V}M_v$. 
The $M_v$ are called the \textit{components} of $M$, and we denote multiplication in both
$M$ and its components by concatenation. It follows from Theorem~\ref{green}
below that the latter embed
naturally in the former, and so there should be no cause for confusion.

If the graph has no edges, $M$ is  the free product of the $M_v$, and at
the other extreme, if the graph is complete, $M$ is their restricted direct
product.

A special case of interest is when all the $M_v$ are isomorphic to the
additive monoid of non-negative integers. The graph product is then
called a \textit{graph monoid} and denoted by $M(\Gamma)$. Graph monoids
are  also  known variously as  \textit{free
partially commutative monoids},  \textit{right-angled Artin monoids}, and
\textit{trace monoids}. These monoids and the corresponding groups have
been extensively investigated (see, for example, \cite{diekert} for
monoids, and \cite{charney} for groups).

Now 
let $X$ be the disjoint union of the $M_v\setminus \{1\}$, and for 
$m \in M_v\setminus \{1\}$ write $C(m) = v$.
We denote the product in the free monoid $X^*$ by $x\circ y$ to 
distinguish it from the products in $M$ and the $M_v$. Clearly
there is a canonical surjective homomorphism $\sigma : X^* \to M$ so
that each element $a$ of $M$ can be represented by an element of
$X^*$, called an \textit{expression} for $a$. 
If  $x_1 \circ x_2 \dots \circ x_n \in
X^*$ is an expression for $a \in M$, the $x_i$ are the
\textit{components} of the expression, and if $C(x_i) = v$, then $x_i$ is
a $v$-component.  If $x_i$ and $x_{i+1}$ are both
$v$-components,
then we may obtain a shorter expression for $a$ by, in the terminology
of \cite{hermiller}, \textit{amalgamating} $x_i$ and $x_{i+1}$: if
$x_i,x_{i+1}\in M_v$ and $x_ix_{i+1} = 1$, delete
$x_i\circ x_{i+1}$; otherwise replace it by the single element $y_i$
of $M_v$ where $y_i = x_ix_{i+1}$ in $M_v$.

If $(C(x_j), C(x_{j+1})) \in E$ for some $j$, then we may obtain a
different expression for $a$ by replacing  $x_j\circ x_{j+1}$ by
$x_{j+1}\circ x_j$. Again we follow \cite{hermiller} and call such
 a move a \textit{shuffle}. Two expressions are
\textit{shuffle equivalent} if one can be obtained from the other by a
sequence of shuffles.

A \textit{reduced expression} is an element $x_1 \circ x_2 \dots \circ x_n \in X^*$
which satisfies
\begin{itemize}
\item [$(i)$] whenever $i < j$ and $C(x_i) = C(x_j)$, there exists
$k$ with $i < k < j$ and $(C(x_i), C(x_k)) \notin E$.
\end{itemize}

Notice that no amalgamation is possible in a reduced expression, and
that a shuffle of a reduced expression is again a reduced expression.
The following is the monoid version of a result of Green \cite{green}
which can also be deduced easily from \cite[Theorem~6.1]{costa1}.

\begin{thm} \label{green}
Every element of $M$ is represented by a reduced expression. Two reduced
expressions represent the same element of $M$ if and only if they are
shuffle equivalent.
\end{thm}

The \textit{length} of an expression is its length as an element of the
free monoid $X^*$; it is clear that shuffle equivalent expressions have
the same length, and so, in view of the theorem, all reduced expressions
representing a given element of $M$ have the same length. We shall use
this observation without further comment, but we note that it also
allows us to define the \textit{length} of an element of $M$ to be the
length of any reduced expression representing it. As an easy consequence of the
notion of length we have the following corollary which we record for
later use. First, we recall that a subset $U$ of a monoid $M$ is
\textit{right unitary} in $M$ if for all elements $m\in M$ and $u\in U$
we have $m\in U$ if $mu \in U$. There is a dual notion of \textit{left
unitary}, and $U$ is \textit{unitary} in $M$ if it is both right and left 
unitary. 

\begin{cor}\label{unitary}
Each $M_v$ is a unitary submonoid of $M$.
\end{cor}
\begin{proof}
If $c\in M_v$, $a\in M$ and $ac\in M_v$, then $ac$ must have length 1
(or zero) and it follows that $a\in M_v$. Thus $M_v$ is right unitary in
$M$, and similarly, it is left unitary.
\end{proof}

%%%%%%%%%%%%%%%%%%%%%%%%%%%%%%%%%%%%%%%%%%%%%%%%%%%%%%%%%%%%%%%%%%%%%%%%

It is natural to ask how properties of $M$ are related to the
corresponding properties of the $M_v$. Several such questions are
considered in \cite{costa1,costa2,fohry}. Our interest is in right
cancellative monoids which do not seem to have been studied in this
context. If $M$ is right cancellative, then so too are the $M_v$ since
they are submonoids of $M$.  Our first aim is to show the converse, that
is, if all the $M_v$ are right cancellative, then so is $M$. Towards
this end we introduce the following terminology.

Let $a, a' \in M$, $v \in V$ and $c \in M_v \setminus \lbrace 1 \rbrace$.
We say that $a$ has \textit{final $v$-component} $c$ and
\textit{final $v$-complement} $a'$ if $a$ admits a
reduced expression $a_1 \circ a_2 \circ \dots \circ a_m \circ c$ 
such that $a_1 a_2 \dots a_m = a'$. We say that $a$ has
\textit{final $v$-component $1$} and \textit{final $v$-complement $a$} if $a$ has
a reduced expression $a_1 \circ \dots \circ a_m$ such that either
\begin{itemize}
\item[(i)] $C(a_j) \neq v$ for all $j$; or
\item[(ii)] there exists $k$ with $(C(a_k), v) \notin E$ and $C(a_j) \neq v$ for all $j \geq k$.
\end{itemize}

Of course, we may define the dual notions of \textit{initial $v$-component}  and
\textit{initial $v$-complement} in the obvious way.

\begin{prop} \label{comp}
For each vertex $v$, each element of $M$ has exactly one final $v$-component
and exactly one final $v$-complement.
\end{prop}
\begin{proof}
For existence, suppose $x \in M$ and let
$$a_1 \circ \dots \circ a_m$$
be a reduced expression for $x$. If conditions (i) or (ii) apply, then,
by definition, $x$ has final $v$-component $1$ and final $v$-complement $x$.
Otherwise, there is a largest integer $j$  with $C(a_j) = v$. If $(C(a_k), v) \notin E$
for some $k > j$, then condition (ii) holds. Hence $(C(a_k), v) \in E$
for all $k > j$, and it follows easily that one can shuffle $a_j$ to the end to
obtain a reduced expression
$$a = a_1 \circ \dots \circ a_{j-1} \circ a_{j+1} \circ \dots \circ a_m \circ a_j$$
so that $x$ has final $v$-component $a_j$ and final $v$-complement
$a_1 \dots a_{j-1} a_{j+1} \dots a_m$.

For uniqueness, suppose first for a contradiction that $x$ has distinct final $v$-components $1$
and $d \neq 1$. Then $x$ has reduced expressions
$a = a_1 \circ \dots \circ a_m$ and $b = b_1 \circ \dots \circ b_n \circ d$ where either
\begin{itemize}
\item[(i)] $C(a_j) \neq v$ for all $j$; or
\item[(ii)] there exists $k$ with $(C(a_k), v) \notin E$ and $C(a_j) \neq v$ for all $j > k$.
\end{itemize}
By 
Theorem~\ref{green}, $b$ can be obtained from $a$ by a sequence of shuffles.
 But clearly in case (i) such a shuffle can never
introduce a $v$-component, while in case (ii) no such shuffle can
change the fact that there exists $a_k$ with $(C(a_k), v) \notin E$ and
$C(a_j) \neq v$ for all $j > k$. Since $b$ does not satisfy either of
the conditions (i) or (ii), this gives a contradiction.

Suppose now that $x$ has reduced expressions
$$a = a_1 \circ \dots \circ a_m \circ c$$ and $$b = b_1 \circ \dots \circ b_m \circ d$$ where
$c,d \in M_v$, $c \neq 1$, $d \neq 1$.
By Theorem~\ref{green}, $b$ can be obtained from
$a$ by a sequence of shuffles. It is clear that no such
shuffle can change the value of the last $v$-component, so we must
have $c = d$.

We now turn our attention to showing that final $v$-complements are unique.
If the (unique) final $v$-component of $x$ is $1$ then by definition we have
that $x$ is the (unique) final $v$-complement of itself, so there is nothing
to prove.
So suppose $x$ has final $v$-component $c \neq 1$, and that there are
reduced expressions
$$a = a_1 \circ \dots \circ a_m \circ c$$ and $$b = b_1 \circ \dots \circ b_m \circ c$$
for $x$. Now by Theorem~\ref{green}, there is a sequence of shuffles
which takes $a$ to $b$. Clearly just by removing those applications which
involve the final $v$-component $c$ of the word, we obtain a sequence of shuffles
 which can be applied to $a_1 \circ \dots \circ a_m$ to yield
$b_1 \circ \dots \circ b_m$. Since these expressions are reduced, it follows by 
Theorem~\ref{green} again that $a_1 \circ \dots \circ a_m$ and $b_1 \circ \dots \circ b_m$ represent
the same element. Thus, $x$ has exactly one final $v$-complement.
\end{proof}

\begin{lem} \label{component lemma}
Let $a \in M$ and $c \in M_v$. 
Suppose $a$ has final $v$-component $d$ and final $v$-complement $a'$.
Then $ac$ has final $v$-component $dc$ and final $v$-complement $a'$.
\end{lem}
\begin{proof}
Suppose first that $a$ has final $v$-component $d \neq 1$. Then $a$ has 
a reduced expression of the form
\begin{equation}\label{eq:red}
a_1 \circ a_2 \circ \dots \circ a_m \circ d
\end{equation}
where $a_1 \circ \dots \circ a_m$ is a reduced expression for $a'$. If
$dc \neq 1$ then clearly
$$a_1 \circ a_2 \circ \dots \circ a_m \circ (dc)$$
is a reduced expression for $ac$, from which the required result is
immediate. On the other hand, if $dc = 1$ then
$$a_1 \circ a_2 \circ \dots \circ a_m$$ is a reduced expression for $ac = a'dc = a'$.
It follows easily from the fact that \eqref{eq:red} is reduced that either
this expression contains no $v$-components, or there exists $k$ such that
$(C(a_k), v) \notin E$ and $a_j \notin v$ for all $j \geq k$. Thus, $ac$
has final $v$-component $1$ and final $v$-complement $a'$, as required.

Now consider the case in which $a$ has final $v$-component $d = 1$. Then
$a$ has a reduced expression
$$a_1 \circ a_2 \circ \dots \circ a_m$$
where $a = a' = a_1 a_2 \dots a_m$ and either 
\begin{itemize}
\item[(i)] $C(a_j) \neq v$ for all $j$; or
\item[(ii)] there exists $k$ with $(C(a_k), v) \notin E$ and $C(a_j) \neq v$ for all $j \geq k$.
\end{itemize}
In both cases, it is easy to check that $a_1 \circ a_2 \circ \dots \circ a_m \circ c$ is a
reduced expression for $ac$, from which it follows that
$ac$ has final $v$-component $dc = c$ and final $v$-complement $a = a'$
as required.
\end{proof}

\begin{thm} \label{cancellative}
A graph product of right [respectively left, two-sided] cancellative monoids
is right [respectively left, two-sided] cancellative.
\end{thm}
\begin{proof}
We prove the result for right cancellative monoids. The corresponding
result for left cancellative monoids is proved similarly using initial
$v$-components and complements, and the result for cancellative monoids
is an immediate consequence of the one-sided results.

 First
observe that, since the graph product monoid is generated by elements from
the embedded components it suffices to show that elements of the embedded
components are right cancellable, that is, that $ac = bc$ implies $a = b$
whenever $c$ belongs $ M_v$ for some $v \in V$.

Suppose that $a$ and $b$ have (unique) final $v$-components $d$ and $e$
respectively, and (unique) final $v$-complements $a'$ and $b'$ respectively.
Then by the preceding lemma, $ac$ has final $v$-component $dc$ and final
$v$-complement $a'$, while $bc$ has final $v$-component $ec$ and final
$v$-complement $b'$.

Since $ac = bc$, we deduce from Proposition~\ref{comp} that $dc = ec$ and
$a' = b'$. But $d$, $e$ and $c$ lie $M_v$ which by assumption is right
cancellative, so we deduce that $d = e$, and hence that $a = a' d = b' e = b$
as required to complete the proof.
\end{proof}

We next consider the question of whether a graph product of monoids each
of which is embeddable in a group is itself embeddable in a group. A
positive answer is a consequence of the next proposition which gives a
universal property defining the graph product. 
We retain the notation of this section. 

\begin{prop} \label{homomorphisms1}
Let $N$ be a monoid and suppose that for each $v\in V$ there is  a
homomorphism $\f_v:M_v \to N$ such that 
$$ (x\f_v)(y\f_u) = (y\f_u)(x\f_v) \text{ for all } (u,v) \in E \text{
and all } x\in M_v, y\in M_u. \qquad (*)$$ 
Put $M=\Gamma_{v\in V}M_v$.      Then there
 is a unique homomorphism $\f:M \to N$ such that $x\f = x\f_v$ for all
$x\in M_v$ and all $v\in V$.
\end{prop}
\begin{proof}
For each $v\in V$, let $\langle A_v \mid R_v\rangle$ be a presentation
for $M_v$, and let $\langle A \mid R \rangle$ be the presentation for
$M$ as at the beginning of the section. Let $\h:A\to N$ be the function
given by $a\h = a\f_v$ where $M_v$ is the unique monoid containing $a$.
Since each $\f_v$ is a homomorphism, $\h$ respects the relations in each
$R_v$, and by hypothesis, $\h$ also respects all the other relations in
$R$. Hence there is a unique homomorphism $\f :M\to N$ which restricts
to $\h$ on $A$ and hence to $\f_v$ on each $M_v$.
\end{proof}

An immediate consequence is the first part of the following result.

\begin{prop} \label{homomorphisms2}
Let $\G$ be a graph, $V$ its set of vertices and 
$\{ M_v\}_{v\in V},\{ N_v\}_{v\in V}$ families of monoids. Let 
$M=\Gamma_{v\in V}M_v$ and $N=\Gamma_{v\in V}N_v$. Then,
given homomorphisms $\f_v:M_v \to N_v$ for each $v\in V$, there is
a unique homomorphism $\f : M \to N$ such that $m_v\f = m_v\f_v$ for all
$v\in V$.

Moreover, if each $\f_v$ is injective, then so is $\f$.
\end{prop}
\begin{proof}
All that remains is to prove the final paragraph. 
Let $a,b \in M$ with $a\f = b\f$ and suppose that $a,b$ have
reduced expressions $a_1 \circ \dots \circ a_m$ and $b_1 \circ \dots
\circ b_n$ respectively where $a_i \in M_{u_i}$ and $b_j \in M_{v_j}$.
Then 
$$  (a_1\f_{u_1}) \dots (a_m\f_{u_m}) = a\f = b\f = (b_1\f_{v_1}) \dots
(b_n\f_{v_n})$$
and since the $\f_v$ are injective, we have that both 
$(a_1\f_{u_1}) \circ \dots \circ (a_m\f_{u_m})$ and 
$(b_1\f_{v_1}) \circ \dots \circ (b_n\f_{v_n})$ are reduced expressions
for $a\f$. Hence they are shuffle equivalent so
that $m = n$ and for some permutation $\s$ we have 
$a_i\f_{u_i} = b_{i\s}\f_{v_{i\s}}$ for all $i$. Since $\im \f_v
\subseteq N_v$ for all $v$, we see that
 $u_i = v_{i\s}$ for each $i$, and so $a_i = b_{i\s}$  since
$\f_{u_i}$ is injective. It  is now clear 
that $a_1 \circ \dots \circ a_m$ and $b_1 \circ \dots \circ b_n$
are shuffle equivalent so that $a = b$ and hence $\f$ is injective.
\end{proof}

The following corollary, which can also be easily proved directly, is
now immediate. 
\begin{cor} \label{groupembedding}
Let $\Gamma$ be a graph with vertex set $V$. If for each $v\in V$, the 
monoid  $M_v$ is embeddable in a group $G_v$, then the graph product
$\Gamma M_v$ is embeddable in the group $\Gamma G_v$.
\end{cor}

 In the next
section we use ideas about inverse hulls to demonstrate  another result about
the closure of a class of right cancellative monoids under graph
products. Specifically we  consider  right
cancellative monoids which satisfy the condition that
  the intersection of two principal left ideals is either
principal or empty.
A right cancellative monoid satisfying this condition is called a \textit{left
LCM monoid}. We show that a graph product 
of left LCM monoids is again a left LCM monoid. 

The reason for the terminology which is borrowed from ring theory is
that the defining condition may also be expressed in terms of 
divisibility. For a right cancellative monoid $C$ and $a,b\in C$, we say that $a$ is a
\textit{left multiple} of $b$ (and that $b$ is a \textit{right factor or divisor}
of $a$) if $a=cb$ for some $c\in C$. If $m$ is is a left multiple of
both  $b$ and $d$, we say it is 
a \textit{common left
multiple} of these elements, and such a common left multiple $m$ is
a \textit{least common left multiple} (LCLM) of $b$ and $d$  if 
every common left multiple of $b$ and $d$ is a left
multiple of $m$.  Equivalently, $m$ 
 is an LCLM of $b$ and $d$ if and only if 
$$ Cb \cap Cd = Cm.$$
Least common left multiples are sometimes known as left least common
multiples. We note that
 a left LCM monoid is a right cancellative monoid in which any two
elements having a common left multiple have an LCLM.

In ring theory (see \cite{beauregard}) an integral domain  (not
necessarily commutative) is called a \textit{left LCM domain} if the
intersection of any two principal left ideals is principal. Thus an
integral domain $R$ is a left LCM domain if and only if the cancellative
monoid of its non-zero elements is a left LCM monoid.

Similarly, one defines \textit{common right factors} and \textit{highest common right
factors}  (HCRF). An element $d$ of $C$ is an HCRF of $a$ and $b$ in $C$
if and only if $Cd$ is the least upper bound of $Ca$ and $Cb$ in the
partially ordered set of principal left ideals of $C$.

We remark that  LCLMs and HCRFs are not uniquely determined in general
being defined only up to
left multiplication by a unit. 

If $C$ is actually cancellative, \textit{common right multiple, common left
factor}, LCRM and HCLF are defined symmetrically.

Examples of right cancellative LCM monoids abound:  the
right locally Garside monoids of Dehornoy \cite{dehornoy1} which, as he
points out include all Artin monoids and all Garside monoids; from ring
theory, we have already mentioned the multiplicative monoid of non-zero elements of any LCM domain. 
 Examples of LCM monoids which are right cancellative
 but  not left cancellative are provided by principal left ideal
right cancellative  monoids; specific examples are the monoids of ordinal
numbers  less than $\omega^{\a}$  (where $\a$ is any ordinal number
greater than 1) under the dual of the usual
operation of ordinal addition.

\section{Inverse hulls} \label{inverse hulls}

With any right cancellative monoid $C$, one can associate an inverse monoid
called the inverse hull of $C$. Before giving the definition
we recall some of the basic concepts
of inverse monoids. For more on the general theory of inverse monoids
see \cite[Chapter~5]{howie} and \cite{lawson}.

An \emph{inverse monoid} is a monoid $M$ such that for all $a\in M$
there is a unique $b\in M$ such that $aba = a$ and $bab = b$. The
element $b$ is the \emph{inverse} of $a$ and is denoted by $a^{-1}$.
It is worth noting that $(a^{-1})^{-1} =a$ and $(ab)^{-1} =
b^{-1}a^{-1}$ for all $a,b \in M$.
The set of idempotents $E(M)$ of $M$ forms a commutative submonoid,
referred to as the \emph{semilattice of idempotents} of $M$. In fact, a
monoid $M$ is an inverse monoid if and only if $E(M)$ is a 
commutative submonoid and for every $a \in M$, there is an element $b\in
M$ such that $aba =a$ (that is, $M$ is regular). 

An \emph{inverse submonoid} of an inverse monoid $M$ is simply a
submonoid $N$
closed under taking inverses.

For a non-empty set $X$, a \emph{partial
permutation} is a bijection $\s:Y\to Z$ for some subsets $Y,Z$ of $X$.
We allow $Y$ and $Z$ to be empty so that the empty function is regarded
as a partial permutation. The set of all partial permutations of $X$ is made
into a monoid by using the usual rule for composition of partial
functions; it is called the \emph{symmetric
inverse monoid} on $X$ and denoted by $\II_X$. That it is an inverse
monoid follows from the fact that if $\s$ is a partial permutation of
$X$, then so is its inverse  (as a function) $\s^{-1}$, and this is the
inverse of $\s$ in $\II_X$ in the sense above. 
  The idempotents of $\II_X$ are
 the \emph{partial identities} $\e_Y^{}$ for all subsets $Y$ of $X$
where $\e_Y^{}$ is the identity map on the subset $Y$.
It is clear that, for $Y,Z \subset X$, we have 
$\e_Y^{}\e_Z^{} = \e_{Y\cap Z}^{}$ and hence that $E(\II_X)$ is
isomorphic to the Boolean algebra of
all subsets of $X$.  

The concept of an inverse hull was introduced by Rees \cite{rees1} to
give an alternative proof of Ore's theorem about the existence of a
group of fractions of a left (or right) Ore cancellative monoid $C$. The
name was introduced in \cite{cp61}, where
the inverse hull of a right cancellative semigroup $C$ is defined. 
A detailed study of the inverse hull is carried out in
\cite{cherubini} where the authors use a definition slightly different
from that in \cite{cp61}. However,  the
two definitions coincide in the case of inverse hulls of right
cancellative monoids, the
only case that we consider. 

After defining what we mean by an inverse hull and recalling some general results,  we
show that a graph product of left LCM monoids
is also a left LCM monoid, and continue
by finding a presentation for the inverse hull of a such a graph product
in terms of presentations for its constituent monoids.
 As a special case  we obtain a
presentation of the inverse hull of a graph monoid.

\subsection{Generalities about inverse hulls} \label{generalities}

As well as being significant in the question of
embeddability in a group, the inverse hull of a right cancellative semigroup is also
important in describing the structure of bisimple, 0-bisimple, simple
and 0-simple inverse semigroups.

Let $C$ be a right cancellative monoid. 
For an element $a$ of $C$, the mapping $\r_a$ with domain $C$ defined by
$$  x\r_a = xa$$
is the \textit{inner right translation} of $C$ determined by $a$. It is
injective  since $C$ is right cancellative, and so it can be regarded as a
member of $\II_C$.   The inverse submonoid of $\II_C$ generated by all the inner
right translations of $C$ is the \textit{inverse hull} $IH(C)$ of $C$.
The inverse of $\r_a$ is, of course, the partial map $\r_a^{-1}:Ca \to
C$, so if $C$ is not a group, then $IH(C)$ contains maps which are not
total. 

The mapping $\eta:C\to IH(C)$ given by $a\eta = \r_a$ is an embedding
of $C$ into $IH(C)$. Moreover, $C\eta$ is the right unit subsemigroup of
$IH(C)$, that is, it consists of those elements $\r \in IH(C)$ for
which there is an element $\t$ with $\r\t = 1_C$. The group of units of
$IH(C)$ is $G\eta$ where $G$ is the group of units of $C$. The left unit
submonoid $L$ of $IH(C)$ consists of the elements $\r_c^{-1}$ for $c\in
C$. For notational convenience, we introduce a left cancellative monoid
$C^{-1}$ containing $G$ as its group of units and such that there is an
anti-isomorphism $c\mapsto c^{-1}$ from $C$ to $C^{-1}$. Here if 
$c\in G$, then $c^{-1}$ is its inverse in $G$, and if $c \notin G$, then
$c^{-1}$ is a new symbol. We can now extend $\eta$ from $G$ to an
isomorphism, also denoted by $\eta$, from $C^{-1}$ to $L$ given by
$c^{-1}\eta = \r_c^{-1}$.

We remark that if $C$ is a group, then every inner right
translation is a permutation of $C$ and $\eta$ is just the Cayley
representation of  $C$.

The empty mapping $\emptyset$  is sometimes a member of $IH(C)$. When it is, it is the
zero of $IH(C)$.  For ease of
expression of some results, we often state them in terms of $IH^0(C)$,
where we define $IH^0(C)$ to be the submonoid $IH(C)\cup \{\emptyset\}$
of $\II_C$.

Clearly, if $a_1,\dots,a_n,b_1,\dots, b_n$ are elements of $C$, then $\r
= \r_{a_1}\r_{b_1}^{-1}\dots \r_{a_n}\r_{b_n}^{-1}$ is a member of
$IH(C)$. It is easy to verify that every element of $IH(C)$ can be
expressed in this way (see \cite[Lemma~2.5]{cherubini}) using the fact
that if $a,b\in C$, then $\r_a\r_b =\r_{ab}$ and $\r_a^{-1}\r_b^{-1} =
\r_{ba}^{-1}$.  Thus every element can be written in the form 
$(a_1^{}\eta)(b_1^{-1}\eta)\dots (a_n^{}\eta)(b_n^{-1}\eta)$.

It is noted in \cite{cp61} that the inverse hull of an infinite cyclic monoid
$\{ x\}^*$ is the bicyclic monoid. 
This example was generalised by  Nivat and Perrot in \cite{np} where
they introduced  polycyclic monoids as the inverse hulls of free monoids.
They give several characterisations of polycyclic monoids, and in
particular, show that the polycyclic monoid $P_X$ on a set $X$
with more than one element has
the following  presentation as a monoid with zero:
 $$ \langle X \cup X^{-1} \mid xx^{-1}=1,xy^{-1}=0 \text{ for } x\neq
y\ (x,y\in X)\rangle.$$
More information on polycyclic monoids can be found in
\cite[Chapter~9]{lawson} and \cite{meakinsapir}.

An independent study of the inverse hull of the free monoid on an
arbitrary nonempty set $X$ was carried out in \cite{knox} where Knox
describes it as a  Rees quotient of a semidirect product of a
semilattice by the free group on $X$.

Further examples of inverse hulls are calculated in \cite{mcmc}.

We recall that a compatible partial  order called the natural partial order is
defined on any inverse semigroup $S$ by the rule that $a\leq b$ if $a
=eb$ for some idempotent $e$. For later use, we characterise this
relation between certain elements of an inverse hull in the following
well known lemma. See \cite{lawson2} for a version of this and its corollary.

\begin{lem} \label{leq}
Let $C$ be a right cancellative monoid and let $a,b,c,d\in C$. Then in
$IH(C)$, 
$$ \r_a^{-1}\r_b \leq \r_c^{-1}\r_d \text{ if and only if } a=xc \text{ and } b=xd
\text{ for some } x\in C.$$ 
\end{lem} 
\begin{proof}
If $\r_a^{-1}\r_b \leq \r_c^{-1}\r_d$, then $a\in \dom \r_a^{-1}\r_b$,
so $a\in \dom \r_c^{-1}\r_d$, that is, $a\in Cc$, say $a=xc$. Then 
$$ b= a\r_a^{-1}\r_b = a\r_c^{-1}\r_d = xd.$$

Conversely, $$\r_a^{-1}\r_b = \r_c^{-1}\r_x^{-1}\r_x\r_{d} \leq 
\r_c^{-1}\r_d.$$ 
\end{proof}

\begin{cor} \label{equal}
Let $C$ be a right cancellative monoid and let $a,b,c,d\in C$. Then in
$IH(C)$, 
$$ \r_a^{-1}\r_b = \r_c^{-1}\r_d \text{ if and only if } a=uc \text{ and } b=ud
\text{ for some unit } u\in C.$$ 
\end{cor}
\begin{proof}
By Lemma~\ref{leq}, there are elements $x,y\in C$ such that $a=xc,\, 
b=xd,\,  c=ya$ and $d=yb$. Hence $a=xya$ and by right cancellation,
$1=xy$. It follows that $x$ and $y$ are units. 
\end{proof}

Recall that in any monoid $M$, Green's relation $\mathscr{R}$ is defined
by the rule that $a\mathscr{R} b$ if and only if $aM = bM$. The relation
$\mathscr{L}$ is the left-right dual of $\mathscr{R}$; we define 
$\mathscr{H} = \mathscr{R} \cap \mathscr{L}$ and $\mathscr{D} =
\mathscr{R}\ \vee \mathscr{L}$. In fact, by \cite[Proposition~2.1.3]{howie}, 
$\mathscr{D} = \mathscr{R} \circ \mathscr{L} = \mathscr{L} \circ \mathscr{R}$.
Finally, $a\J b$ if and only if $MaM =
MbM$. In an inverse monoid, $a\mathscr{R} b$ if and only if $aa^{-1} =
bb^{-1}$ and similarly, $a\mathscr{L}b$ if and only if $a^{-1}a =
b^{-1}b$. In $\II_X$, we have $\r\R\s$ if and only if $\dom \r = \dom
\s$, and $\r\L\s$ if and only if $\im \r = \im \s$
\cite[Exercise~5.11.2]{howie}. The following lemma thus follows immediately
from \cite[Proposition~3.2.11]{lawson}.

\begin{lem}
Let $C$ be a right cancellative monoid. Then, for elements  $\r,\s$ of
$IH^0(C)$, 
\begin{enumerate}
\item [(1)]  $\r\R\s$ in $IH^0(C)$ if and only if $\dom \r = \dom \s$,
\item [(2)] $\r\L\s$ in $IH^0(C)$ if and only if $\im \r = \im \s$.
\end{enumerate}   
\end{lem}

We mention that $\L$ is a right congruence and
$\R$ is a left congruence. More information on Green's relations can be found in
\cite{howie,lawson}. Finally, an inverse monoid (or semigroup) is
$0$-\textit{bisimple} if all its non-zero elements are $\D$-related; it
is \textit{bisimple} if all its elements are $\D$-related. Thus if $a,b$
are nonzero elements of a $0$-bisimple inverse monoid $M$, then there are
elements $c,d\in M$ such that $a\L c\R b$ and $a\R d \L b$.

In \cite{np}, it is pointed out that the equivalence of $(1)$ and $(3)$
in the next proposition can be obtained by slightly modifying the theory
of Clifford \cite{clifford}. 
A proof of the whole result can be extracted from \cite{mcal}, but for
the convenience of the reader and completeness we give an elementary proof.

\begin{prop} \label{0-bisimple}
The following are equivalent for a right cancellative monoid $C$:
\begin{enumerate}
\item [$(1)$]  $IH^0(C)$ is $0$-bisimple,
\item [$(2)$] The domain of each non-zero element of $IH^0(C)$ is a
principal left ideal,
\item [$(3)$] $C$ is a left LCM monoid,
\item [$(4)$]  Every non-zero element of $IH^0(C)$ can be written in the form
$\r_c^{-1}\r_d$ for some $c,d\in C$.
\end{enumerate}
\end{prop}
\begin{proof}
Suppose that $(1)$ holds, and let $\r$ be a non-zero element of
$IH^0(C)$. Then $\r$ is $\D$-related to the identity, and so
$\R$-related to an element $\s$ of the left unit submonoid. Hence
$\dom\r= \dom\s$ and since  $\s
=\r_a^{-1}$ for some $a\in C$, we have $\dom\r = Ca$ so that $(2)$ holds.

If $(2)$ holds, and $a,b\in C$, then since $Ca\cap Cb$ is the
domain of $\r_a^{-1}\r_a\r_b^{-1}\r_b$, we see that $Ca \cap Cb$ is
either principal or empty. Thus $(3)$ holds.

Now suppose that $(3)$ holds and let $\r$ be a non-zero element of
$IH^0(C)$. We have noted that 
$\r = \r_{a_1}\r_{b_1}^{-1}\dots \r_{a_n}\r_{b_n}^{-1}$ for some
$a_i,b_i\in C$, and so it is enough to show that if $c,d\in C$ and
$\r_c\r_d^{-1}$ is non-zero, then for some $a,b\in C$ we have 
$\r_c\r_d^{-1} = \r_a^{-1}\r_b$. Now the domain of $\r_c\r_d^{-1}$ is 
$(Cc \cap Cd)\r_c^{-1}$, and by assumption, $Cc \cap Cd = Cs$ for some $s\in
C$. Thus $s=rc=td$ for some $r,t\in C$ and an easy calculation shows
that $\r_c\r_d^{-1} =  \r_r^{-1}\r_t$.

Finally, if $(4)$ holds, let $\r = \r_a^{-1}\r_b$ be a non-zero element
of  $IH^0(C)$. Now $\r_a^{-1}$ is $\L$-related to the identity, and
since $\L$ is a right congruence, we get $\r\L\r_b$. But $\r_b\R1$, so
$\r$ is $\D$-related to the identity, and $(1)$ follows.
\end{proof}

It is worth noting that if $C$ is a left LCM monoid, then the product of 
two non-zero elements in $IH^0(C)$ is given by
$$ (\r_a^{-1}\r_b)(\r_c^{-1}\r_d) = 
\begin{cases} 0 \qquad \qquad \qquad \: \text{ if } Cb \cap Cc = \emptyset\\
             \r_{sa}^{-1}\r_{td} \qquad \qquad \text{ if }  Cb \cap Cc =  Csb = Ctc.
   \end{cases}
$$

Although it is not relevant to the present paper, it is worth noting
that every $0$-bisimple inverse monoid $M$ is isomorphic to $IH^0(C)$ where
$C$ is the right unit submonoid of $M$ \cite{np}, so that the
preceding proposition applies to all such monoids. 
We make use of the proposition to prove the next theorem for
which we
also need the following lemma.

\begin{lem} \label{intersection1}
Let $\G = (V,E)$ be a graph and, for each $v\in V$, let $C_v$ be a right
cancellative monoid, and $C=  \G_{v\in V}C_v$. 
Let $c,d$ be nonunits in $C_v,C_u$ respectively
where $(u,v)\in E$. Then 
$$ Cc \cap Cd = Ccd.$$
\end{lem}
\begin{proof}
Since $(u,v)\in E$, we have $cd = dc$ so that $Ccd \sub Cc \cap Cd$. Now
suppose that $a \in Cc \cap Cd$ so that $a = sc = td$ for some $s,t \in
C$. By Lemma~\ref{component lemma}, $a$ has final $v$-component $c'c$
and final $u$ component $d'd$ where $c'$ is the final $v$-component of
$s$ and $d'$ is the final $u$-component of $t$. Neither $c'c$ nor $d'd$
can be $1$ since $c,d$ are not units. Thus $a$ has reduced
expressions $x_1\circ \dots \circ x_n \circ (c'c)$ and $y_1 \circ \dots
\circ y_n \circ (d'd)$ which, by Theorem~\ref{green}, must be shuffle
equivalent. Hence one of the $x_i$, say $x_j$, must be $d'd$ and one can
shuffle it to the end to obtain a reduced expression
$$ x_1 \circ \dots \circ x_{j-1} \circ x_{j+1} \circ \dots \circ x_n \circ
(c'c) \circ (d'd)$$
for $a$. Hence $a = x_1 \dots  x_{j-1} x_{j+1} \dots x_n (c'c)(d'd)$,
and since $c\in C_v,d'\in C_u$ so that $cd' = d'c$ (as $(u,v)\in E$) we
have 
$$a = x_1 \dots  x_{j-1} x_{j+1} \dots x_n c'd'cd \in Ccd$$
completing the proof.
\end{proof}

\begin{thm} \label{Lpreserved}
Let $\G = (V,E)$ be a graph and, for each $v\in V$, let $C_v$ be a left
LCM monoid. Then the graph product
$C = \G_{v\in V}C_v$ is also a left LCM monoid.
\end{thm}
\begin{proof}
We have that $C$ is right cancellative by Theorem~\ref{cancellative}. To
prove that $C$ is a left LCM monoid, we show that every non-zero element of $IH^0(C)$ can
be written in the form $\r_a^{-1}\r_b$ for some $a,b \in C$, and appeal
to Proposition~\ref{0-bisimple}.

We claim that if $c,d\in C$ and $\t = \r_c\r_d^{-1}$ is non-zero, then $\t=
\r_a^{-1}\r_b$ for some $a,b\in C$. 

The result follows from this claim and our earlier observation that every
non-zero element of $IH^0(C)$ can be written in the form 
$\r_{a_1}\r_{b_1}^{-1}\dots \r_{a_n}\r_{b_n}^{-1}$.

We note that the claim is true if one of $c,d$ is a unit: if $r=c^{-1}$
exists, then 
$$ \t = \r_{r^{-1}}\r_d^{-1} = \r_r^{-1}\r_d^{-1} = \r_{dr}^{-1} = \r_{dr}^{-1}\r_1,$$
and if $d$ is a unit, then
$$\r_c \r_d^{-1} = \r_c \r_{d^{-1}} = \r_{cd^{-1}} = \r_1^{-1}\r_{cd^{-1}}.$$

We now assume that $c,d$ are both nonunits and continue by proving the
claim in the case when $c$ has length 1, that
is, $c\in C_v$ for some $v\in V$. Suppose that
 $d$ has length 1. If $d\in
C_v$, then $\t = \r_{a}^{-1}\r_b$ since $C_v$ is a left LCM monoid.
Let $d\in C_u$ with $u \neq v$. 
If $(u,v) \notin E$;  then no reduced expression
 ending in $c$ is shuffle equivalent to one ending in $d$ and  it
follows that
$Cc \cap Cd = \emptyset$. Thus $\t = \emptyset$, a contradiction. Hence $(u,v) \in
E$ so that $cd = dc$. By Lemma~\ref{intersection1}, $Cc \cap Cd =
Ccd$. It follows that $\dom \r_c\r_d^{-1} = Cd = \dom \r_d^{-1}\r_c$,
and it is easily verified that $\r_c\r_d^{-1} = \r_d^{-1}\r_c$. Hence
the claim holds for all $c$ and $d$ of length 1; in fact, we have $\r_c\r_d^{-1}
= \r_a^{-1}\r_b$ where $a$ and $b$ also have length 1.

To complete the proof, let $c,d\in C$ have reduced expressions $c_1 \circ \dots
\circ c_h$ and $d_1 \circ \dots \circ d_k$ so that $\r_c\r_d^{-1} =
\r_{c_1}\dots \r_{c_h}\r^{-1}_{d_1}\dots \r^{-1}_{d_k}$. Now apply the
case for $n=1$ repeatedly.
\end{proof}

In the next lemma we compare intersections of principal left ideals in
the graph product and in its component monoids.

\begin{lem} \label{intersection2}
Let $\G = (V,E)$ be a graph and, for each $v\in V$, let $C_v$ be a left
LCM monoid and let
$C = \G_{v\in V}C_v$. If $x,y\in C_v$ for some $v\in V$, then 
$$ C_v x \cap C_v y = \emptyset \text{ if and only if } Cx \cap Cy
=\emptyset.$$
Moreover, if $C_v x \cap C_v y = C_v z$, then $Cx \cap Cy = Cz$. 
\end{lem}
\begin{proof}
Clearly, if $Cx \cap Cy=\emptyset$, then $C_v x \cap C_v y = \emptyset$.
Conversely, suppose that $ax = by$ for some $a,b\in C$. Let $a$ and $b$
have final $v$-components $c$ and $d$ respectively. Then by Lemma~\ref{component
lemma}, $ax$ has final $v$-component $cx$ and $by$ has final $v$-component
$dy$. But $ax = by$, so by Proposition~\ref{comp}, 
$cx = by \in C_vx \cap C_vy$.

Suppose that $C_v x \cap C_v y = C_v z$; then certainly, $Cz
\subseteq Cx \cap Cy$. If $r= ax = by$ for some $a,b \in C$, then
applying Lemma~\ref{component lemma} and  Proposition~\ref{comp} again
we see that $r$ has final $v$-component $cx = dy$ where $c$ and $d$ are
the final $v$-components of $a$ and $b$ respectively. Thus 
$cx \in C_v x \cap C_v y$ so $cx= mz$ for some $m\in C_v$, and if $r'$
is the final $v$-complement of $r$, then $r=r'mz \in Cz$ as required.
\end{proof}

We are now in a position to prove the following result which will be
important in the next subsection.

\begin{prop} \label{embeddedinversehull}
If $C$ is the graph product $\G_{v\in V}C_v$ ofleft LCM
monoids $C_v$, then, for each $v\in V$,
the inverse hull
$IH^0(C_v)$ is embedded in $IH^0(C)$. 
\end{prop}
\begin{proof}
For $x\in C_v$ denote the inner right translations of $C_v$ and $C$ determined by $x$
by $\r_x$ and $\d_x$ respectively. Non-zero elements of $IH^0(C_v)$ have
the form $\r_x^{-1}\r_y$ and so we can define $\h : IH^0(C_v) \to IH^0(C)$ by
$0\h = 0$ and $(\r_x^{-1}\r_y)\h = \d_x^{-1}\d_y$. 

To see that $\h$ is well defined, suppose that $\r_x^{-1}\r_y = \r_z^{-1}\r_t$.
Then by Corollary~\ref{equal},
$x=uz$ and $y=ut$ for some unit $u$ of $C_v$. Certainly $u$ is a unit of
$C$, so we have $\d_x^{-1}\d_y = \d_z^{-1}\d_t$ as required.

To see that $\h$ is injective, suppose that $\d_x^{-1}\d_y =
\d_z^{-1}\d_t$ where $x,y,z,t \in C_v$. 
Then by  Corollary~\ref{equal}, we have $x=qz$ and
$y=qt$  for some unit $q$ of $C$. By Corollary ~\ref{unitary},
 $C_v$ is unitary in $C$, and since
$qt,t\in C_v$, we have $q\in C_v$. It is easy to see that $q^{-1}$ is also
in $C_v$, so that $q$ is a unit of $C_v$ and so 
$\r_x^{-1}\r_y = \r_z^{-1}\r_t$ as required.

Finally, we show that $\h$ is a homomorphism. Let 
$\r_x^{-1}\r_y , \r_z^{-1}\r_t$ be elements of $IH^0(C_v)$. 

If $C_v y \cap C_v z = \emptyset$, then by Lemma~\ref{intersection2}, 
$Cy \cap Cz = \emptyset$. From the rule for multiplication following
Proposition~\ref{0-bisimple}, we have
$(\r_x^{-1}\r_y)(\r_z^{-1}\r_t) =0$, and since, by
Theorem~\ref{Lpreserved}, $C$ is left LCM, we also have 
$(\d_x^{-1}\d_y)(\d_z^{-1}\d_t) =0$. 

If $C_v y \cap C_v z \neq \emptyset$, then since $C_v$ is an LCM monoid,
 we have $C_v y \cap C_v z = C_va$ for some $a\in C_v$, say
$a=ry=sz$ where $r,s\in C_v$. By
Lemma~\ref{intersection2}, we also have $Cy \cap Cz = Ca$, and so by the
rule for multiplication we see that
$$ (\r_x^{-1}\r_y)(\r_z^{-1}\r_t) = \r^{-1}_{rx}\r^{-1}_{st}$$
and 
$$(\d_x^{-1}\d_y)(\d_z^{-1}\d_t) = \d^{-1}_{rx}\d^{-1}_{st}.$$

It follows that $\h$ is a homomorphism as required.

\end{proof}

\subsection{Inverse hulls of graph products of left LCM monoids}
Let $\G=(V,E)$ be a graph and $\{ C_v\}_{v\in V}$ be a family of 
left LCM monoids. Let
$C = \G_{v\in V}C_v$ be the graph product
of the $C_v$; we
have just proved that $C$ is also a left LCM monoid. In this 
subsection our first
goal is to find a presentation (as a monoid with zero) 
for $IH^0(C)$ in terms of given
presentations for the inverse monoids $IH^0(C_v)$.

We begin by establishing some notation. 
Let $D$ be any right cancellative monoid with group of units $G$ and let $Y$
be a symmetric set of monoid generators for $G$ (i.e., $y\in Y$ if and
only if $y^{-1} \in Y$). We assume that $1\notin Y$ and 
take $Y$ to be empty if $G = \{1\}$.
Let $X$ be a set of nonunits in $D$ such that $X \cup Y$ generates $D$. 
Let $X^{-1} = \{ x^{-1}:x \in X\}$ be a set disjoint from $X$  such
that $x \mapsto x^{-1}$ is a bijection, and $X^{-1} \cup Y$ generates
the left cancellative monoid $D^{-1}$ anti-isomorphic to $D$. Since any  element of
$IH(D)$ can be written in the form 
$ \r_{a_1}\r_{b_1}^{-1}\dots \r_{a_n}\r_{b_n}^{-1}$, it follows that
there is a homomorphism from the free monoid $(X \cup X^{-1} \cup Y)^*$
onto $IH(D)$ which sends $x$ to $\r_x$, $y$ to $\r_y$ and $x^{-1}$ to
$\r_x^{-1}$.  Thus $IH(D)$ has a presentation of the form 
$\langle X \cup X^{-1} \cup Y \mid R \rangle$ for some set of relations
$R$. 
We can also regard $\langle X \cup X^{-1} \cup Y \mid R \rangle$
 as a presentation for $IH^0(D)$ in the
class of monoids with zero. Since $\r_x\r_x^{-1} =1$ for all $x\in X$,
we can assume that $xx^{-1} = 1$ is a relation in $R$ for every $x\in
X$. Similarly, since $\r_y$ is a unit for all $y \in Y$, we can assume
that we have relations $yy^{-1} = 1 = y^{-1}y$ in $R$ for all $y\in Y$.

Turning to the graph product $C = \G_{v\in V}C_v$ we note that we have a
corresponding graph product $C^{-1} = \G_{v\in V}C_v^{-1}$ of the left
cancellative monoids $C_v^{-1}$. Writing $G_v$ for the common group of
units of $C_v$ and $C_v^{-1}$, we remark that, by
\cite[Proposition~7.1]{costa1}, the common group of units of $C$ and
$C^{-1}$ is $G= \G_{v\in V}G_v$. We also observe that the
anti-isomorphisms between the $C_v$'s and the $C^{-1}_v$'s extend, by a
slight variation of Proposition~\ref{homomorphisms2} to an
anti-isomorphism between $C$ and $C^{-1}$.
Now put  
$S_v = IH^0(C_v)$ for  each $v\in V\!\!,$ and let 
$\langle X_v \cup X_v^{-1} \cup Y_v \mid R_v \rangle$ be a presentation for
$S_v$ of the type described in the previous paragraph. It will be
convenient to adopt the following notation convention: 
  $x_v, y_v$
denote  elements of $X_v, Y_v$ respectively; $t_v$ denotes an element of
$X_v \cup Y_v$ and $z_v$ denotes any element
of $Z_v = X_v \cup X_v^{-1} \cup Y_v$.

We now put $X =
\bigcup_{v\in V}X_v$, $X^{-1} = \bigcup_{v\in V}X_v^{-1}$, 
$Y = \bigcup_{v\in V}Y_v$, and $Z = X \cup X^{-1} \cup Y$. 
As in Section~\ref{graph products}, we will want to consider the free
monoid on $\bigcup_{v\in V}C_v$ as well as the free monoid 
$Z^*$. To avoid confusion about the various
products,  we write $\circ$, as before, for the product in the former free
monoid, and $\diamond$ for that in $Z^*$.

 Next, we introduce several sets of relations
amongst words over $X \cup X^{-1} \cup Y$ (and zero) as follows:

\begin{itemize}
\item [(1)] $R= \bigcup_{v\in V}R_v$;
\item [(2)] $N = \{ x_v\diamond y_{u_1} \diamond \dots\diamond y_{u_m}
\diamond x_w^{-1} = 0 : m \geq 0, \forall\  x_v \in X_v,x_w \in X_w,\\
\hspace*{1cm} y_{u_i} \in Y_{u_i} \text{ with }(v,w)\notin E \text{ and }
v \neq w \}$;
\item [(3)]  $\Com\, = \{  z_u\diamond z_v = z_v\diamond z_u : \forall\ 
 z_u \in Z_u,z_v \in Z_v \text{ with }(u,v) \in E \}$.
\end{itemize}

The \textit{polygraph product} of the $S_v$ is defined to be the monoid 
 $\PG = \PG_{v\in V}(S_v)$ given by the
presentation 
$$\langle Z \mid R  \cup N  \cup \Com \rangle.$$

There is thus a surjective homomorphism  $\zeta: Z^* \to \PG$.
For each $v\in V$, the generators and relations of $IH^0(C_v)$ are among
those for $\PG$ and so there is a 
monoid homomorphism $\psi_v$ from $IH^0(C_v)$ into $\PG$ determined by
$\r_{t_v}\psi_v = t_v\zeta$ and $\r_{x_v}^{-1}\psi_v = x_v^{-1}\zeta$
for $t_v\in X_v \cup Y_v$ and $x_v \in X_v$.

The right unit
submonoid of $IH^0(C_v)$ is isomorphic to $C_v$ via the map $\eta_v:C_v \to
IH^0(C_v)$ given by $c\eta_v = \r_c$. As noted in the preceding
subsection, we can also extend $\eta_v$ from $G_v$ (the group of units
of $C_v$) to the left cancellative monoid $C^{-1}_v$ to give an
isomomorphism onto the left unit submonoid of $IH^0(C_v)$.  Composing
$\eta_v$ with the restriction of $\psi_v$ first to the right unit
submonoid of $IH^0(C_v)$, then to the left unit submonoid, we obtain
monoid homomorphisms from $C_v$ and $C^{-1}_v$ into $PG$ both of which we
denote by $\h_v$. There is no ambiguity here since these homomorphisms
agree on the common group of units of $C_v$ and $C^{-1}_v$. We observe
that if $c_v = t_1\dots t_n$ where $t_i\in X \cup Y$, then
 $$ 
c_v\h_v =  (t_1\eta_v\psi_v) \dots
(t_n\eta_v\psi_v) = \r_{t_1}\psi_v \dots \r_{t_n}\psi_v = t_1\zeta\dots
t_n\zeta = (t_1 \diamond \dots \diamond t_n)\zeta,
$$
and 
$$
c_v^{-1}\h_v = (t_n^{-1}\dots t_1^{-1})\h_v = 
\r_{t_n}^{-1}\psi_v \dots \r_{t_1}^{-1}\psi_v =
t_n^{-1}\zeta \dots t_1^{-1}\zeta = 
(t_n^{-1} \diamond \dots \diamond t_1^{-1})\zeta.
$$

Now by
Proposition~\ref{homomorphisms1} and its dual, there are unique
homomorphisms from $C$ into the  right unit
submonoid of $\PG$, and from $C^{-1}$ into the left unit submonoid of
$\PG$ which restrict to $\h_v$ on each $C_v$ and $C_v^{-1}$
respectively. We have noted that the common group of units of $C$ and
$C^{-1}$ is $G = \G_{v\in V}G_v$ where $G_v$ is the common group of
units of $C_v$ and $C_v^{-1}$. As no non-units are in both $C$ and
$C^{-1}$, there is no ambiguity in denoting both homomorphisms by $\h$.

From the above we see that the squares

\begin{center}
\begin{pspicture}(-1,-1)(4,3) %\showgrid
\psset{nodesep=4pt}
\rput(0,2){\rnode{A}{$(X \cup Y)^*$}}
\rput(3,2){\rnode{B}{$C$}}
\rput(0,0){\rnode{C}{$(X\cup X^{-1} \cup Y)^*$}}
\rput(3,0){\rnode{D}{$\PG$}}
\ncline[arrowsize=3pt 2.5]{->}{A}{B}
\ncline[arrowsize=3pt 2.5]{->}{A}{C}\Bput{$\iota$}
\ncline[arrowsize=3pt 2.5]{->}{C}{D}\Aput{$\zeta$}
\ncline[arrowsize=3pt 2.5]{->}{B}{D}\Aput{$\h$}
\end{pspicture}
\qquad \qquad
\begin{pspicture}(-1,-1)(4,3) %\showgrid
\psset{nodesep=4pt}
\rput(0,2){\rnode{A}{$(X^{-1} \cup Y)^*$}}
\rput(3,2){\rnode{B}{$C^{-1}$}}
\rput(0,0){\rnode{C}{$(X\cup X^{-1} \cup Y)^*$}}
\rput(3,0){\rnode{D}{$\PG$}}
\ncline[arrowsize=3pt 2.5]{->}{A}{B}
\ncline[arrowsize=3pt 2.5]{->}{A}{C}\Bput{$\iota$}
\ncline[arrowsize=3pt 2.5]{->}{C}{D}\Aput{$\zeta$}
\ncline[arrowsize=3pt 2.5]{->}{B}{D}\Aput{$\h$}
\end{pspicture}
\end{center}
are commutative where $\iota$ is the inclusion map. It follows that
every non-zero element of $\PG$ can be written in the form
$(a_1\h)(b_1^{-1}\h)\dots (a_k\h)(b_k^{-1}\h)$ where $a_i,b_i\in C$. In
fact, we can do better than this as we see in the next lemma.

\begin{lem} \label{crucial}
Every non-zero element of $\PG = \PG_{v\in V}(S_v)$ can be written
in the form
$(a^{-1}\h)(b\h)$ where $a,b\in C$.
\end{lem}
\begin{proof}
In view of the remark preceding the lemma, it is enough to show that if
$c,d\in C$, then either $(c\h)(d^{-1}\h) =0$ or $(c\h)(d^{-1}\h) = 
(a^{-1}\h)(b\h)$ for some $a,b\in C$. This is clearly true if $c$ or $d$
is a unit of $C$, so we may assume that neither is a unit.

We use induction on the length, as defined in Section~\ref{graph
products}, of $c$ and $d$. We start by considering $d$ of length 1, and
proving  by induction on the length of $c$ that for any $c\in C$, 
either $(c\h)(d^{-1}\h) =0$
or $(c\h)(d^{-1}\h) = (a^{-1}\h)(b\h)$ for some $a,b\in C$ with $a$ of
length 1. First, suppose that $c$ has length 1. Then $c\in C_u,d\in C_v$
for some $u,v$. If $u=v$, then 
$$
(c\h)(d^{-1}\h) = (c\h_u)(d^{-1}\h_u) = (\r_c\psi_u)(\r_d^{-1}\psi_u) =
(\r_c\r_d^{-1})\psi_u.
$$
Since $C_u$ is left LCM, we have, by Proposition~\ref{0-bisimple},
 that $\r_c\r_d^{-1}$ is
either zero or equal to $\r_a^{-1}\r_b$ for some $a,b \in C_u$.  Hence,
if non-zero, 
$$
(c\h)(d^{-1}\h) = (\r_c\r_d^{-1})\psi_u = (\r_a^{-1}\r_b)\psi_u =
(\r_a^{-1}\psi_u)(\r_b\psi_u) = (a^{-1}\h)(b\h). 
$$

If $u\neq v$, let $c = t_1'\dots t_m'$ and $d = t_1\dots t_n$ where
$t_i'\in X_u \cup Y_u$ and $t_j \in X_v \cup Y_v$. If $(u,v) \in E$,
then $t_i'\diamond t_j = t_j\diamond t_i'$ is a relation in $\Com$ for all $i,j$ and it
follows that $(c\h)(d^{-1}\h)  = (d^{-1}\h)(c\h)$. 

Suppose that $(u,v) \notin E$.  Since $c,d$ are
non-units, not all the $t_i'$ are units and not all the $t_j$ are units.
Let $h$ and $k$ be the largest integers such $t_h'$ and $t_k$ are
non-units. Then we can write $x_h'$ for $t_h'$ and $x_k$ for  $t_k$, and
similarly, we can write $y_i'$ for $t_i'$ when $i > h$ and $y_j$ for
$t_j$ when $j > k$. Consider 
$(x_h'\diamond y_{h+1}' \diamond \dots \diamond y_m' \diamond y_n^{-1}
\diamond \dots \diamond 
y_{k+1}^{-1} \diamond x_{k+1}^{-1})\zeta$. This element is zero (by virtue of the
relations in $N$) and so $(c\h)(d^{-1}\h) = 0$.

Thus our claim is true for all $c$ and $d$ of length 1. Now suppose that
for any $c,d\in C$ with $c$ of length less than $m$ and $d$ of length 1,
we have $(c\h)(d^{-1}\h) = 0$ or $(c\h)(d^{-1}\h) = (a^{-1}\h)(b\h)$ for
some $a,b \in C$ with $a$ of length 1.

Next, let $c\in C$ have length $m$, say $c_1 \circ \dots c_m$ is a
reduced expression for $c$, and let $d\in C_v$. By the current induction
assumption, $(c_2\dots c_m\h)(d^{-1}\h)$ is either zero or can be written in
the form $(a^{-1}\h)(b\h)$ with $a$ of length 1. In the former case, it
is clear that $(c\h)(d^{-1}\h) = 0$. In the latter case, if
$(c\h)(d^{-1}\h)$ is non-zero, we have 
\begin{align*}
  (c\h)(d^{-1}\h) &= ((c_1\dots c_m)\h)(d^{-1}\h) = (c_1\h)((c_2\dots
c_m)\h)(d^{-1}\h)\\
  &= (c_1\h)(a^{-1}\h)(b\h)\\
 &= (a_1^{-1}\h)(b_1\h)(b\h) = (a_1^{-1}\h)((b_1b)\h)
\end{align*}
where $a_1$ has length 1, using the fact that $c_1$ and $a$ both have length 1.

Thus we have proved our claim that for any $c,d\in C$ with $d$ of length 1,
either $(c\h)(d^{-1}\h) =0$
or $(c\h)(d^{-1}\h) = (a^{-1}\h)(b\h)$ for some $a,b\in C$ with $a$ of
length 1.

Now assume inductively that for any $c\in C$ and any $d\in C$ of length
$n-1$, if $(c\h)(d^{-1}\h) \neq 0$, then 
$(c\h)(d^{-1}\h) = (a^{-1}\h)(b\h)$ for some $a,b\in C$.
Let $d\in C$ have a reduced expression $d_1 \circ \dots \circ d_n$ so
that 
\begin{align*}
 (c\h)(d^{-1}\h) &= (c\h)(d_n^{-1}\h)((d_{n-1}^{-1}\dots d_1^{-1})\h) \\
       &= (a_1^{-1}\h)(b_1\h)((d_{n-1}^{-1}\dots d_1^{-1})\h) \text{ for some }
a_1,b_1\in C  \text{ (by the case for $n=1$)}\\
 &= (a_1^{-1}\h)((b_1\h)(d_{n-1}^{-1}\dots d_1^{-1})\h) \\
 & = (a_1^{-1}\h)(a_2^{-1}\h)(b_2\h) \text{ for some }
a_2,b_2\in C  \text{ (by the induction assumption)}\\
&= (a_1^{-1}a_2^{-1})\h(b_2\h)  \\
 &= (a^{-1}\h)(b\h) \text{ where } a =a_2a_1 \text{ and } b = b_2.
\end{align*}
This completes the proof of the lemma.

\end{proof}

We now consider $IH^0(C)$. We remind the reader that (as a monoid with
zero) each $IH^0(C_v)$ is generated by 
$\{ \r_{x_v}, \r_{x_v}^{-1}, \r_{y_v} : x_v\in X_v, y_v\in Y_v\}$  and that
 $IH^0(C)$ is generated by $Q =\{ \r_x, \r_x^{-1}, \r_y : x\in X, y\in Y\}$
where $X = \bigcup_{v\in V}X_v$, $X^{-1} = \bigcup_{v\in V}X_v^{-1}$ and 
$Y = \bigcup_{v\in V}Y_v$. As before, we also assume that $R_v$ is a set
of defining relations for $IH^0(C_v)$ and put $R = \bigcup_{v\in V}R_v$.

\begin{lem} \label{relations1}
With respect to the generating set $Q$, the relations in $R$ are
satisfied by $IH^0(C)$.
\end{lem}
\begin{proof}
By Proposition~\ref{embeddedinversehull},
$IH^0(C_v)$ is embedded in $IH^0(C)$ for all $v\in V$. The
relations in $R$ are relations in $R_v$ for some $v$, so hold in $IH^0(C_v)$
and hence in $IH^0(C)$.
\end{proof}

\begin{lem} \label{relations2}
With respect to the generating set $Q$, the relations in $N$ are
satisfied by $IH^0(C)$.
\end{lem}
\begin{proof}
Suppose that 
$x_v\diamond y_{u_1}\diamond \dots\diamond  y_{u_m}\diamond x_w^{-1} =0$
is a relation in $N$
so that $(v,w)\notin E$ and $v\neq w$. Then in $IH^0(C)$ we have
$$\dom \r_{x_v}\r_{y_{u_1}}\dots \r_{y_{u_m}}\r_{x_w}^{-1} = 
(Cx_vy_{u_1}\dots y_{u_m} \cap Cx_w)(\r_{x_v}\r_{y_{u_1}}\dots
\r_{y_{u_m}})^{-1}.$$

Since $x_v$ is not a unit and $(v,w)\notin E$, in an expression for an
element $a$ of $Cx_vy_{u_1}\dots y_{u_m}$, any amalgamation involving
$x_v$ produces a non-unit of $C_v$, so a non-unit of $C_w$ cannot be
shuffled to the end of the expression. Hence the final $w$-component of
$a$ is a unit. But the final $w$-component of an element of $Cx_w$ must be a
left multiple of $x_w$ and hence be a non-unit. It follows from
Proposition~\ref{comp} that
$Cx_vy_{u_1}\dots y_{u_m} \cap Cx_w = \emptyset$ and so 
$\r_{x_v}\r_{y_{u_1}}\dots \r_{y_{u_m}}\r_{x_w}^{-1} =0$.
\end{proof}

\begin{lem} \label{relations3}
With respect to the generating set $Q$, the relations in $\Com$ are
satisfied by $IH^0(C)$.
\end{lem}
\begin{proof}
Following our convention that $t_u,x_u$ denote arbitrary elements of $X_u
\cup Y_u$ and $X_u$ respectively, relations in $\Com$ have one of the forms:
\begin{itemize}
\item [$(i)$] $t_u\diamond t_v = t_v\diamond t_u$;
\item [$(ii)$] $x_u\diamond x_v^{-1} = x_v^{-1}\diamond x_u$;
\item [$(iii)$] $x_u^{-1}\diamond x_v^{-1} = x_v^{-1}\diamond x_u^{-1}$
\end{itemize}
where $(u,v) \in E$. Relations of the form $(i)$ are satisfied in
$IH^0(C)$ since 
$$ \r_{t_u}\r_{t_v} = \r_{t_ut_v}  = \r_{t_vt_u} = \r_{t_v}\r_{t_u}.$$

Consider a relation as in $(ii)$. 
By Lemma~\ref{intersection1} we have 
$Cx_u \cap Cx_v = Cx_ux_v$, and since $x_ux_v = x_vx_u$ in $C$, we have
 $$\dom \r_{x_u}\r_{x_v}^{-1} = (\im \r_{x_u} \cap 
\dom \r_{x_v}^{-1})\r_{x_u}^{-1} = (Cx_vx_u)\r_{x_u}^{-1} = Cx_v.$$
Similarly, we calculate  $\im \r_{x_u}\r_{x_v}^{-1} = Cx_u$. 

Since $\im \r_{x_v}^{-1} = C = \dom \r_{x_u}$, it is easy to see that we
also have 
$\dom \r_{x_v}^{-1}\r_{x_u} = Cx_v$ and $\im \r_{x_v}^{-1}\r_{x_u} =
Cx_u$, and it follows that
$\r_{x_u}\r_{x_v}^{-1} = \r_{x_v}^{-1}\r_{x_u}$.

Finally consider a relation of the form $(iii)$. In this case, since
$(u,v)\in E$, we also have that $x_u\diamond x_v = x_v\diamond x_u$ is a
relation in $\Com$. Hence $ \r_{x_v}\r_{x_u}  = \r_{x_u}\r_{x_v}$
follows by $(i)$, and since $IH^0(C)$ is an inverse monoid,
$$ \r_{x_u}^{-1}\r_{x_v}^{-1} = (\r_{x_v}\r_{x_u})^{-1} = (\r_{x_u}\r_{x_v})^{-1}
 = \r_{x_v}^{-1}\r_{x_u}^{-1}.$$
\end{proof}

We now use the lemmas together to obtain the following theorem where we
retain the notation of this section.

\begin{thm}  \label{presentation}
The monoids $\PG_{v\in V}(S_v)$ and $IH^0(C)$ are isomorphic.
\end{thm}
\begin{proof}
Consider the function  $\b:X\cup X^{-1} \cup Y \to IH^0(C)$ given by
$x\b = \r_x$, $x^{-1}\b = \r_x^{-1}$ and $y\b = \r_y$. It follows from
Lemmas~\ref{relations1} to \ref{relations3} that $\b$ extends to a
homomorphism, again denoted by $\b$, from $\PG$ to $IH^0(C)$. Since
the latter is generated by $Q$,
the homomorphism is surjective.

Let $r,s\in \PG$ and suppose that $r\b = s\b$. By
Lemma~\ref{crucial}, $r = (a^{-1}\h)(b\h)$ and $s = (c^{-1}\h)(d\h)$ for some $a,b,c,d
\in C$. Hence $((a^{-1}\h)(b\h))\b = ((c^{-1}\h)(d\h))\b$ so that $\r_a^{-1}\r_b =
\r_c^{-1}\r_d$, and hence by Corollary~\ref{equal}, there is a unit $e$
of $C$ such that $c =ea$ and $d= eb$. If
$m,n\in C$, then there are correponding elements $m^{-1},n^{-1}$ in
$C^{-1}$ and $(mn)^{-1} = n^{-1}m^{-1}$. Thus, using the fact that $e$ is a
unit in $C$,
\begin{align*}
 s &= (c^{-1}\h)(d\h) = ((ea)^{-1}\h)((eb)\h)\\
 &= (a^{-1}e^{-1})\h(eb)\h = (a^{-1}\h)(e^{-1}\h)(e\h)(b\h)\\
&=  (a^{-1}\h)((e^{-1}e)\h)(b\h) = (a^{-1}\h)(b\h)\\ 
 &= r.
\end{align*}

Thus $\b$ is an isomorphism and the proof is complete.
\end{proof}

\section{Polygraph monoids} \label{polygraph monoids}

Theorem~\ref{presentation} gives us a presentation for $IH^0(C)$ and also allows us to
write the elements of $\PG$ in the form $a^{-1}b$ with $a,b \in C$ where
$a^{-1}b = c^{-1}d$ if and only if $c =ea$ and $d= eb$ for some unit $e$
of $C$. The presentation simplifies
considerably in the case when each $C_v$ (and hence also $C$) has
trivial group of units, in that  $Y = \emptyset$ and consequently
$$N = \{ x_u \diamond x_v^{-1} = 0 : \forall\ x_u\in X_u, x_v\in X_v
\text{ with }(u,v) \notin E \text{ and } u\neq v \}.$$
Thus we have the presentation 
$$ \langle X \cup X^{-1} \mid R  \cup N \cup \Com \rangle$$
for $IH^0(C)$.  

A particular instance of this is when each $C_v$ is a
free monogenic monoid. Then $S_v = IH^0(C_v)$ is the bicyclic monoid
with zero adjoined, and as a monoid with zero it has the 
presentation with two generators: 
$\langle x_v,x_v^{-1} \mid x_vx_v^{-1} = 1\rangle$. In this case, the graph
product of the $C_v$ is a graph monoid $M(\G)$ with presentation 
$$\langle x_v\  (v\in V) \mid  x_ux_v = x_vx_u \text{ if } (u,v) \in E
\rangle.$$

The monoid $IH^0(M(\G))$ is called a \textit{polygraph monoid} and we
denote it by $P(\G)$. 
 Put $X = \{ x_v : v\in V\}$ and for $x\in C_u$, $y \in C_v$,
write $x \sim y$ if   $(u,v) \in E$, and abusing
notation, write $x \nsim y$ to mean $u \neq v$ and $(u,v) \notin
E$. Then our polygraph monoid has a presentation
\begin{align*}
\langle X \cup X^{-1} \mid  xx^{-1} &= 1;\ xy^{-1} = 0   \text{ if } x \nsim y;\\ 
 xy &= yx,\: xy^{-1} = y^{-1}x,\: x^{-1}y^{-1} = y^{-1}x^{-1} \text{ if } x \sim y \rangle.
\end{align*}

If $\G$ has no edges, then $M(\G) = X^*$ is the free monoid on $X$ and
the polygraph monoid $IH^0(M(\G))$ is the monoid with presentation
$$\langle X \cup X^{-1} \mid xx^{-1} = 1; xy^{-1} = 0   \text{ if }
x\neq y \rangle,$$
that is, it is the polycyclic monoid introduced in \cite{np} and studied
in, among others, \cite{knox,meakinsapir,lawson}.

Let $P(\G)$ be the polygraph monoid determined by the graph $\G= (V,E)$.
Since $P(\G)$ is the inverse hull (with zero adjoined if necessary) 
of the graph monoid $M(\G)$, it follows from
the remarks following Theorem~\ref{presentation} 
that every non-zero element of $P(\G)$ can be
written as $a^{-1}b$ for some $a,b\in M(\G)$. Since the
identity is the only unit in $M(\G)$ it follows 
 that if $a,b,c,d \in M(\G)$, then
 $a^{-1}b = c^{-1}d$  if and only if $a=c$ and $b=d$. Thus
we may regard the non-zero elements of $P(\G)$ as pairs $(a,b)$ where
$a,b\in M(\G)$. With
this notation, the product in $P(\G)$ is given by
$$ (a,b)(c,d) = \begin{cases} 0 \quad \text{ if } M(\G)b \cap M(\G)c = \emptyset\\
             (sa,td) \text{ if }  M(\G)b \cap M(\G)c =  M(\G)sb = M(\G)tc.
   \end{cases}
$$

\begin{prop} \label{polygraphmonoid1}
The monoid $P(\G)$ is a $0$-bisimple (bisimple if it has no zero) inverse
monoid with 
$$ E(P(\G)) = \{ (a,a) : a\in M(\G) \} \cup \{0\} $$
as its set of idempotents.
\end{prop}
\begin{proof}
Since graph monoids left LCM, Proposition~\ref{0-bisimple}
gives  that $P(\G)$
is a $0$-bisimple (bisimple if it has no zero) inverse monoid. 

It is easy to verify that any element of the form $(a,a)$ is
idempotent. Suppose that $(a,b)(a,b) = (a,b)$. 
Then $(ta,sb) = (a,b) $ where $M(\G)a \cap  M(\G)b = M(\G)sb = M(\G)ta$. Hence, by
the criterion for equality, $ta = a$ and $sb = b$ in $M(\G)$ so that $t = s
= 1$. Thus $M(\G)a = M(\G)b$ and hence $a=b$. 
\end{proof}

Since $P(\G)$ is 0-bisimple, $\D = \J$ and two elements are 
$\D$-related if and only if they are both non-zero or both equal to
zero. In the next proposition we characterise the other Green's
relations on $P(\G)$.

\begin{prop} \label{polygraphmonoid2}
For elements $(a,b),(c,d)$ of $P(\G)$,
\begin{itemize}
\item [$(1)$] $(a,b)^{-1} = (b,a)$;
\item [$(2)$] $(a,b)\L (c,d)$ if and only if $b=d$;
\item [$(3)$] $(a,b)\R (c,d)$ if and only if $a=c$;
\item [$(4)$] $\H$ is trivial.
\end{itemize}
\end{prop}
\begin{proof} $(1)$ is an easy calculation.
In an inverse monoid, elements $s,t$ are $\L$-related if and only if
$s^{-1}s = t^{-1}t$. Using this and (1) we see that in $P(\G)$ we have 
$(a,b)\L  (c,d)$ if and only if $b=d$. 

The result for $\R$ is similar, and then it 
follows that $\H$ is
trivial.
\end{proof}

We next consider the properties  of being $E^*$-unitary or strongly
$E^*$-unitary. 
For any inverse monoid $S$, the semilattice of idempotents of $S$ is denoted by $E(S)$,
and if $S$ has a zero, $E^*(S)$ denotes the set of non-zero idempotents.
Recall from Section~\ref{graph products} 
that a subset $U$ of $S$ is \textit{right unitary} in $S$ if for $u\in
U$, $s\in S$ we have $su \in U$ if and only if $s\in U$. There is a dual
notion of \textit{left unitary}, and if $U$ is both left and right
unitary, it is said to be \textit{unitary} in $S$. If $U$ is either $E(S)$ or $E^*(S)$,
then it is left unitary if and only if it is right unitary. We say that  
 $S$ is $E$-\textit{unitary} if $E(S)$ is a unitary subset of $S$, and
that 
it is $E^*$-\textit{unitary} \cite{maria} (or
$0$-$E$-unitary \cite{meakinsapir}, \cite{lawson}) if $E^*(S)$ is a
unitary subset of $S$. Chapter~9 of \cite{lawson} is devoted to
$E^*$-unitary inverse semigroups.

A special class of $E^*$-unitary inverse semigroups was introduced
independently in \cite{bffg} and \cite{lawson2}. In general, if we
adjoin a zero to a semigroup $S$, we denote the semigroup obtained by $S^0$.
An inverse semigroup $S$ with zero is \textit{strongly $E^*$-unitary} if
there is a group $G$ and a function $\h:S\to G^0$ satisfying:
\begin{enumerate}
\item $a\h =0$ if and only if $a=0$;
\item $a\h =1$ if and only if $a\in E^*(S)$;
\item if $ab\neq 0$, then $(ab)\h = (a\h)(b\h)$. 
\end{enumerate}

Condition (1) says that $\h$ is $0$-\textit{restricted}; conditions~(1)
and~(2) together say that $\h$ is \textit{idempotent-pure}, that is, the
only elements which map to idempotents are idempotents; and
condition~(3) says that $\h$ is a \textit{prehomomorphism}. In general,
prehomomorphisms between inverse monoids are defined in terms of the
natural  order on the monoids, but the general definition is equivalent
to condition~(3) when the codomain is a group with zero adjoined.
Implicit in \cite{bffg} is the result that an inverse semigroup with
zero is strongly $E^*$-unitary if and only if it is a Rees quotient of
an $E$-unitary inverse semigroup. This was made explicit with an easy
proof in \cite{ben}. As well as \cite{bffg} and \cite{ben}, further
information about  strongly $E^*$-unitary inverse
semigroups, including many examples, can be found in the 
surveys \cite{lawson3} and \cite{mcal2}.

%%%%%%%%%%%%%%%%%%%%%%%%%%%%%%%%%%%%%%%%%%%%%%%%%%%%%%%%%%%%%%%%%%%%%%
We are interested in the connection between strongly $E^*$-unitary
inverse monoids and embeddability of cancellative monoids in groups. The
following result is due to Margolis \cite{stuart}; we include a proof
for completeness.

\begin{prop} \label{necandsuff}
Let $S$ be a cancellative monoid. Then $S$ is embeddable in a group if
and only if $IH^0(S)$ is strongly $E^*$-unitary.
\end{prop}
\begin{proof}
Suppose first that $S$ is embedded in a group $G$. As noted in
Section~\ref{generalities}, every (non-zero) element $\r$ of $IH^0(S)$ can
be expressed as $\r_{a_1}\r_{b_1}^{-1}\dots  \r_{a_n}\r_{b_n}^{-1}$ for
some elements $a_1,b_1\dots, a_n,b_n$ of $S$. Define a mapping $\h:IH^0(S)
\to G^0$ by putting $0\h=0$ and 
$(\r_{a_1}\r_{b_1}^{-1}\dots  \r_{a_n}\r_{b_n}^{-1})\h = a_1b_1^{-1}\dots
a_nb_n^{-1}$ if $\r_{a_1}\r_{b_1}^{-1}\dots  \r_{a_n}\r_{b_n}^{-1}$ is
non-zero. 

If $\r = \r_{a_1}\r_{b_1}^{-1}\dots  \r_{a_n}\r_{b_n}^{-1}
= \r_{c_1}\r_{d_1}^{-1}\dots  \r_{c_m}\r_{d_m}^{-1}$ is non-zero, then for every
element $x$ in $\dom \r$, we have 
$$ x\r = x\r_{a_1}\r_{b_1}^{-1}\dots  \r_{a_n}\r_{b_n}^{-1} 
= x\r_{c_1}\r_{d_1}^{-1}\dots  \r_{c_m}\r_{d_m}^{-1}$$
so that in $G$, the following equation holds:
$$ x\r = xa_1b_1^{-1}\dots a_nb_n^{-1} = xc_1d_1^{-1}\dots c_md_m^{-1}$$
and hence $a_1b_1^{-1}\dots a_nb_n^{-1} = c_1d_1^{-1}\dots c_md_m^{-1}$.
Thus $\h$ is well-defined. 

By definition, $\h$ is 0-restricted. If $\r$ is as defined above and $\r\h
=1$, then we have $a_1b_1^{-1}\dots a_nb_n^{-1} =1$ and it follows that $x\r =x$
for all $x\in \dom \r$ so that $\r=I_{\dom\r}$ and $\h$ is idempotent
pure. Finally, it is clear from the definition that if $\r,\s\in
IH^0(S)$ and $\r\s\neq 0$, then $(\r\s)\h = (\r\h)(\s\h)$ so that $\h$
is a prehomomorphism. Thus $IH^0(S)$ is strongly $E^*$-unitary.

For the converse, we suppose that $IH^0(S)$ is strongly $E^*$-unitary and
consider 
  a 0-restricted idempotent pure prehomomorphism $\h:IH^0(S) \to G^0$
from $IH^0(S)$ 
to a group $G$ with zero adjoined. For each $a\in S$, we have the
element $\r_a$ of $IH(S)$, and since $\dom \r_a =S$, it follows that
$\r_a\r_b =\r_{ab}$ for any $a,b\in S$. Since $\h$ is 0-restricted,
$\r_a\h \in G$, and we have 
$$ (\r_a\h)(\r_b\h) = (\r_a\r_b)\h = \r_{ab}\h.$$
Hence  we can define $\psi:S\to G$ by $a\psi = \r_a\h$, and
$(a\psi)(b\psi)=(ab)\psi$, that is, $\psi$ is a homomorphism. It is also
injective, for if $a\psi =b\psi$, then $\r_a\h=\r_b\h$. Now
$\r_a^{-1}\r_b$ is a non-zero element of $IH^0(S)$, and so
$$ (\r_a^{-1}\r_b)\h = (\r_a^{-1}\h)(\r_b\h) = (\r_a^{-1}\h)(\r_a\h) =
(\r_a^{-1}\r_a)\h =1$$
since $\r_a^{-1}\r_a$ is a non-zero idempotent. But $\h$ is idempotent
pure, so $\r_a^{-1}\r_b$ is an idempotent, that is, it is the identity
map on its domain. Hence for $x\in \dom(\r_a^{-1}\r_b)$ we have
$x\r_a^{-1}\r_b = x$. Now $x\r_a^{-1} =u$ where $x=ua$ and also
$x=u\r_b=ub$ so that $ua=ub$ and $a=b$ by cancellation.

Thus $S$ is embedded in $G$.
\end{proof}

It is well known (and a consequence of Corollary~\ref{groupembedding}) 
that there is an embedding $\h:M(\G) \to G(\G)$ of the
graph monoid $M(\G)$ into the graph group $G(\G)$, and so 
we have the following.
\begin{cor} \label{polygraphmonoid3}
For any graph $\G$, the polygraph monoid $P(\G)$ is strongly $E^*$-unitary.
\end{cor}

In the next section, we see that $P(\G)$ has another special property,
namely that it is $F^*$-inverse.

\section{$F^*$-inverse $0$-bisimple inverse monoids} \label{F*-inverse}

Recall that an inverse monoid $S$ is $F^*$-\textit{inverse} if every
non-zero element of $S$ is under a unique maximal element in the natural
partial order. If $S$ does not have a zero, it is said to be
$F$-inverse, and in this case, the definition is equivalent to every
$\s$-class containing a maximum element. (Here $\s$ is the minimum group
congruence on $S$.) However, we shall use the term
$F^*$-inverse to include both cases. It is easy to verify that every
$F^*$-inverse monoid is $E^*$-unitary. An $F^*$-inverse monoid which is
also strongly $E^*$-unitary is called \textit{strongly $F^*$-inverse}.
It follows from Corollary~\ref{polygraphmonoid3} and the results of
this section that a polygraph monoid is strongly $F^*$-inverse.

We find a criterion for a $0$-bisimple inverse monoid with cancellative
right unit submonoid to be
$F^*$-inverse in terms of a property of its right unit submonoid. We
remark that by a result of Lawson~\cite{lawson2}, for a $0$-bisimple
inverse monoid, having a cancellative
right unit submonoid is equivalent to being $E^*$-unitary.

\begin{lem} \label{maximal}
Let $C$ be a right cancellative monoid and suppose that $IH^0(C)$ is
$0$-bisimple. If $a,b\in C$ have only units as common left factors, then
$\r_a^{-1}\r_b$ is maximal in $IH^0(C)$.
\end{lem}
\begin{proof}
Since $IH^0(C)$ is $0$-bisimple, every element has the form
$\r_a^{-1}\r_b$ for some $a,b\in C$. The result is now immediate from
Lemma~\ref{leq} and its corollary.
\end{proof}

If $C$ is a cancellative monoid, we denote the partially ordered set of
principal right (resp. left) ideals by $P_r(C)$ (resp. $P_{\ell}(C)$).
From the remarks at the end of Section~\ref{graph products}, we see that
$P_r(C)$ is a join semilattice if and only if every pair of elements has 
an HCLF, and it is a meet semilattice if and only if  every pair of
elements has an LCRM. Corresponding remarks apply to $P_{\ell}(C)$.

\begin{prop} \label{mainlemma}
Let $C$ be a cancellative monoid and suppose that $IH^0(C)$ is
$0$-bisimple. Then $IH^0(C)$ is $F^*$-inverse if and only if $P_r(C)$
is a join semilattice.
\end{prop}
\begin{proof}
Suppose that every pair of elements of $C$ have a HCLF and let $\a$ be a
non-zero element of $IH^0(C)$. Then $\a = \r_a^{-1}\r_b$ for some $a,b\in
C$. Let $x$ be an HCLF of $a$ and $b$, say $a=xc$ and $b=xd$. Then the
only common left factors of $c$ and $d$ are units, so by
Lemma~\ref{maximal}, $\r_c^{-1}\r_d$ is maximal. But $\r_a^{-1}\r_b \leq
\r_c^{-1}\r_d$ by Lemma~\ref{leq}, so $\a$ lies beneath a maximal
element.

If $\r_a^{-1}\r_b \leq \r_p^{-1}\r_q$  for some $p,q\in C$, then by
Lemma~\ref{leq}, $a=yp$
and $b=yq$ for some $q\in C$. Hence $x=yz$ for some $z\in C$ so that
$a=yp=yzc$ and $b=yq=yzd$. By left cancellation, $p=zc$ and $q=zd$ so
that $\r_p^{-1}\r_q \leq \r_c^{-1}\r_d$ by Lemma~\ref{leq}. Thus
$\r_c^{-1}\r_d$ is the unique maximal element above $\r_a^{-1}\r_b$, and
$IH^0(C)$ is $F^*$-inverse.

Conversely, suppose that $IH^0(C)$ is $F^*$-inverse, and let $a,b\in C$.
Then there is a unique maximal element $\r_c^{-1}\r_d$ above
$\r_a^{-1}\r_b$. By Lemma~\ref{leq}, $a=xc$ and $b=xd$ for some $x\in C$.
If $y$ is a common left factor of $a$ and $b$, then $a=yp$ and $b=yq$
for some $p,q\in C$ so that  $\r_a^{-1}\r_b \leq \r_p^{-1}\r_q$. Now
$\r_p^{-1}\r_q \leq \a$ for some maximal $\a$, and by uniqueness,
$\a=\r_c^{-1}\r_d$. It follows that  $p=zc$ and $q=zd$ for some $z$ so
that $xc= a=yzc$ whence $x=yz$ and $y$ is a left factor of $x$. Thus $x$
is a HCLF of $a$ and $b$.
\end{proof}

An abstract version of this proposition is given in the following
result. 

\begin{prop} \label{E-unitary}
Let $S$ be an $E^*$-unitary $0$-bisimple ($E$-unitary bisimple) inverse
monoid, and let $C$ be its right unit submonoid. Then $S$ is
$F^*$-inverse ($F$-inverse) if and only if $P_r(C)$
is a join semilattice.
\end{prop}
\begin{proof}
Since $S$ is $0$-bisimple, the right unit submonoid $C$ of $S$ is a left
LCM monoid by Proposition~1 of \cite{np}, and from the same proposition
we have that $S$ is isomorphic to $IH^0(C)$.
By Theorem~5 of \cite{lawson2}, $C$ is
cancellative so that the result is now immediate by Proposition~\ref{mainlemma}. 
\end{proof}

A \textit{Garside monoid} is defined to be a cancellative monoid whose
only unit is the identity, that is a lattice with respect to both left
and right divisibility, and that satisfies additional finiteness
conditions (see, for example, \cite{dehornoy}). 
Such monoids have proved to be important in the study of algebraic and
algorithmic properties of braid groups and, more generally, Artin groups
of finite type. We note that if $C$ is a Garside monoid, then since the
identity is the only unit, regarded as a partially ordered set under
left divisibility, $C$ is order-isomorphic to $P_r(C)$ under reverse
inclusion. Thus $P_r(C)$ is a lattice so that $IH(C)$ does not have a
zero, and hence the next corollary follows
immediately from Proposition~\ref{0-bisimple} and
Proposition~\ref{mainlemma}.

\begin{cor}
The inverse hull of a Garside monoid $C$ is a bisimple $F$-inverse monoid.
\end{cor}

We now turn to Artin monoids. Recall that an Artin monoid is a monoid
generated by a non-empty set $X$ subject to relations of the form 
$xyx\dots = yxy\dots $ where $x,y \in X$, both sides of a given relation
have the same length, and at most one such relation holds for each pair 
$x,y\in X$. Thus graph monoids are Artin monoids where both sides of
each defining relation have length 2.
The associated \textit{Artin group} of a given Artin monoid $A$
is the group given by the 
presentation of $A$ regarded as a group presentation. 
Rather than the definition, we use
some of the properties of Artin monoids which we now recall.
 The first three in the list below can be found in \cite{bs}, the third is
also given in \cite{deligne}, and the fourth is from \cite{paris}. Let
$A$ be an Artin monoid. Then we have the following.

\begin{enumerate}
\item [1.] $A$ is cancellative.
\item [2.] The intersection of two principal left (right) ideals of $A$ is
either empty or principal.
\item [3.] $A$ is left (and right) Ore if and only if it is of finite
type. 
\item [4.] $A$ embeds in its associated Artin group.
\end{enumerate}

\begin{prop}
The inverse hull\  $IH(A)$ of an Artin monoid $A$ is strongly $F^*$-inverse.
\end{prop}
\begin{proof}
It follows from Proposition~\ref{necandsuff} and item (4) above that 
$IH(A)$ is strongly $E^*$-unitary
($E$-unitary in case $A$ is of finite type). Moreover, we have already
noted that condition~$(d)$ of Proposition~\ref{0-bisimple} is satisfied.
Hence $IH(A)$ is $0$-bisimple (bisimple if $A$ is of finite type).

Thus by Proposition~\ref{mainlemma}, it is enough to show that any two elements of
$A$ have a HCLF. This is noted in \cite{bs}. The argument is as
follows. Since the defining relations of $A$ are homogeneous (i.e.,
the two words in each relation have the same length), it follows that
any factor (left or right) of an element $w$ of $A$ has length at most 
$|w|$. Hence any element of $A$ has only finitely many left factors.
Let $x_1,\dots,x_k$ be the common left factors of two elements $v$ and
$w$ of $A$. Then by the right handed version of item 2,
$$ x_1A \cap \dots \cap x_kA = xA$$
for some $x$. (that is, $x$ is the least common left multiple of
$x_1,\dots,x_k$.) Now $x$ is a common left factor of $v$ and $w$, (so
must be one of the $x_i$s) and is clearly the HCLF of $v$ and $w$. 
\end{proof}

Since a graph monoid is a special type of Artin monoid, we immediately
have the following corollary.

\begin{cor}
For a graph $\G$, the polygraph monoid is a strongly $F^*$-inverse monoid.
\end{cor}

\section*{Acknowledgements}
The work reported in this paper was started during a visit by the first
author to Carleton University where the second author was conducting
research supported by the  Leverhulme Trust.
 The first author would like to thank Benjamin
Steinberg for making his visit possible, and the School of Mathematics
and Statistics
at Carleton for its hospitality. The paper was completed when the second
author was an RCUK Academic Fellow at the
University of Manchester.

\end{document}